\documentclass[a4paper]{amsart}

\usepackage[applemac]{inputenc}
\usepackage[T1]{fontenc}
\usepackage[english]{babel}
\usepackage{babelbib}
\usepackage{enumerate}
\usepackage{url}
\selectbiblanguage{english}

\DeclareMathOperator{\GL}{GL}
\DeclareMathOperator{\Hom}{Hom}
\DeclareMathOperator{\End}{End}
\DeclareMathOperator{\Lie}{Lie}
\DeclareMathOperator{\Ker}{Ker}
\DeclareMathOperator{\Bor}{Bor}
\DeclareMathOperator{\rad}{rad}
\DeclareMathOperator{\Opp}{Opp}
\DeclareMathOperator{\Par}{Par}
\DeclareMathOperator{\Spec}{Spec}
\DeclareMathOperator{\Stab}{Stab}

\newcommand{\op}{\textup{op}}

\newcommand{\A}{\mathbb{A}}
\newcommand{\F}{\mathbb{F}}
\newcommand{\G}{\mathbb{G}}
\newcommand{\Gm}{\mathbb{G}_m}
\newcommand{\N}{\mathbf{N}}
\renewcommand{\P}{\mathbb{P}}
\newcommand{\Q}{\mathbb{Q}}
\newcommand{\R}{\mathbb{R}}
\newcommand{\V}{\mathbb{V}}
\newcommand{\Z}{\mathbb{Z}}

\renewcommand{\epsilon}{\varepsilon}

\newcommand{\df}{:=}
\newcommand{\too}{\longrightarrow}
\newcommand{\iso}{\simeq}

\begin{document}

\newtheoremstyle{theorem}{11pt}{11pt}{\itshape}{}{\bfseries}{.}{.5em}{}
\newtheoremstyle{note}{11pt}{11pt}{}{}{\bfseries}{.}{.5em}{}

\theoremstyle{theorem}
    \newtheorem{theorem}{Theorem}[section]
    \newtheorem{proposition}[theorem]{Proposition}
    \newtheorem{lemma}[theorem]{Lemma}
    \newtheorem{corollary}[theorem]{Corollary}
    \newtheorem*{claim}{Claim}

\theoremstyle{note}
    \newtheorem{definition}[theorem]{Definition}
    \newtheorem{remark}[theorem]{Remark}
    \newtheorem{example}[theorem]{Example}
    
      \numberwithin{equation}{subsection}

\title[Maximality of hyperspecial compact subgroups]{Maximality of hyperspecial compact subgroups avoiding Bruhat-Tits theory}

\author{Marco Maculan}

\address{Institut Math\'ematique de Jussieu\\ 
4 place Jussieu\\
75005 Paris (France)}

\email{marco.maculan@imj-prg.fr}



\begin{abstract}
Let $k$ be a complete non-archimedean field (non trivially valued). Given a reductive $k$-group $G$, we prove that hyperspecial subgroups of $G(k)$ (\textit{i.e.} those arising from reductive models of $G$) are maximal among bounded subgroups. The originality resides in the argument: it is inspired by the case of $\GL_n$ and avoids all considerations on the Bruhat-Tits building of $G$.
\end{abstract}


\maketitle

\section{Introduction}

\subsection{Background} Over the complex numbers, a connected linear algebraic group $G$ is reductive if and only if it contains a Zariski-dense compact subgroup. If $G$ is semi-simple such a subgroup corresponds to a maximal real Lie subalgebra of $\Lie G$ on which the Killing form is negative definite.

If one replaces the field of complex numbers by the field of $p$-adic ones (or, more generally, any finite extension of it) an analogue characterisation holds: a connected linear algebraic group $G$ is reductive if and only if $G(\Q_p)$ contains a maximal compact subgroup \cite[Propositions 3.15-16]{PlatonovRapinchuk}. In this case, a maximal compact subgroup is of the form $\mathcal{G}(\Z_p)$ for a suitable integral model $\mathcal{G}$ of $G$.

Reversing the logic one might wonder, given an integral model $\mathcal{G}$ of $G$, whether the compact subgroup $\mathcal{G}(\Z_p)$ is maximal. It has to be (according to work of Bruhat, Hijikata, Rousseau, Tits among others) if the special fibre of $\mathcal{G}$ is a reductive group over $\F_p$ -- the associated compact subgroup is then called \textit{hyperspecial}, whence the title of the article. The purpose of the present paper is to expound a proof of this result without using the theory of Bruhat-Tits building (and the combinatorics needed to construct it).

\subsection{Statement of the results} 
In order to be more precise and to state the main theorem in its full generality, let $k$ be a non-archimedean field (that we suppose complete and non-trivially valued), $k^\circ$ its ring of integers and $\tilde{k}$ its residue field. Let $\mathcal{G}$ be a reductive $k^\circ$-group\footnote{Let $G$ be a group $S$-scheme. We say that $G$ is \textit{reductive} (resp. \textit{semi-simple}) if it verifies the following conditions:
\begin{enumerate}
\item $G$ is affine and smooth over $S$;
\item for all $s \in S$, the $\bar{s}$-algebraic group $G_{\bar{s}} \df G \times_S \bar{s}$ is connected and reductive (resp. connected and semi-simple).
\end{enumerate} Here $\bar{s}$ denotes the spectrum of an algebraic closure of the residue field $\kappa(s)$ at $s$. See \cite[XIX, D\'efinition 2.7]{SGA3}.} and $G$ its generic fibre. The main result is the following:

\begin{theorem} \label{Thm:HyperspecialSubgroupIntro} The subgroup $\mathcal{G}(k^\circ)$ is a maximal bounded subgroup of $G(k)$.
\end{theorem}

When $G$ is split this theorem can be deduced from \cite[\S 3.3 and \S8.2]{BruhatTitsI} taking in account Exemple 6.4.16 (b), \textit{loc.cit.}. Note that, under the hypothesis of $G$ being split, Theorem \ref{Thm:HyperspecialSubgroupIntro} is due to Bruhat over a $p$-adic field \cite{BruhatBourbaki, BruhatBruxelles}. The quasi-split case, \textit{i.e.} when $G$ contains a Borel subgroup defined over $k$, is covered by \cite[Th\'eor\`eme 4.2.3]{BruhatTitsII} and the general case, by \cite[Th\'eor\`eme 5.1.2]{RousseauThese} (the existence of the reductive model $\mathcal{G}$ implies that $G$ splits over a non-ramified extension \cite[Theorem 6.1.16]{ConradSGA3}).


\subsection{}
The subgroups of $G(k)$ of the form $\mathcal{G}(k^\circ)$ are called hyperspecial. When the residue field $\tilde{k}$ is finite, the existence of a $k^\circ$-reductive model $\mathcal{G}$ of $G$ is equivalent to $G$ being quasi-split over $k$ and being split over a non-ramified extension \cite[Theorem 2.6]{ConradSmoothRep}. In particular, although maximal compact subgroups always exist for an arbitrary reductive group over a locally compact field, hyperspecial subgroups do not.

Hyperspecial subgroups are anyway crucial objects in the study of representations of $p$-adic groups and, even though Theorem \ref{Thm:HyperspecialSubgroupIntro} is a basic result, all the proofs I am aware of rely on the deep knowledge of the combinatorics of $G(k)$ which comes at the end of Bruhat-Tits theory. More precisely, one sees $\mathcal{G}(k^\circ)$ as the stabiliser of a (hyperspecial) vertex of the Bruhat-Tits building $\mathcal{B}(G, k)$, which is a maximal bounded subgroup. This exploits implicitly that the integral model $\mathcal{G}$ induces a Tits system on $G(k)$ (when $k$ is discretely valued) and a valued root datum of $G(k)$ (when the valuation is dense). 

Instead, the proof of Theorem \ref{Thm:HyperspecialSubgroupIntro} presented here elaborates on the argument for the case $\GL_n$, using tools from algebraic geometry involving flag manifolds of $G$. When $k$ is a $p$-adic field the advantage of the present approach is that it avoids all the computations contained in \cite{Hijikata} needed in order to show that $G(k)$ admits a Tits system. Let us recall that for the general linear group the proof of Theorem \ref{Thm:HyperspecialSubgroupIntro} goes as follows:

\begin{enumerate}
\item Consider the norm $\| (x_1, \dots, x_n) \| = \max \{ |x_1|, \dots, |x_n|\} $ on $k^{n}$. The subgroup of $\GL_n(k)$ of elements letting $\| \cdot \|$ invariant is $\GL_n(k^\circ)$.

\item \smallskip If $H$ is a bounded subgroup containing $\GL_n(k^\circ)$ consider the norm $\| \cdot \|_H$ defined for every $x \in k^n$ by
$ \| x \|_H \df \sup_{h \in H} \| h(x)\|$. The ratio of the norms $\| \cdot \|_H /\|\cdot \|$ gives rise to a well-defined function $ \phi \colon \P^{n-1}(k) \to \R_+$ which  is clearly $\GL_n(k^\circ)$-invariant.

\item \smallskip Since the group $\GL_n(k^\circ)$ acts transitively on $\P^{n-1}(k)$, $\phi$ must be constant. In particular, $H$ is contained in $\GL_n(k^\circ)$.
\end{enumerate}

\subsection{} 
\par The problem with passing from $\GL_{n}$ to an arbitrary reductive $k$-group $G$ is that the latter does not have a canonical representation on which one can consider norms. We prefer to interpret $\P^{n-1}$ as a flag variety of $\GL_n$ and the norm $\| \cdot \|$ as the metric that it induces on the line bundle $\mathcal{O}(1)$. Moreover, we think at the latter as the metric naturally induced by the line bundle $\mathcal{O}(1)$ on $\P^{n-1}$ over the ring of integers $k^\circ$.

When treating the case of an arbitrary reductive $k^\circ$-group $\mathcal{G}$ of generic fibre $G$, this suggests to replace:
\begin{itemize}
\item  the projective space $\P^{n-1}_k$ by the variety $X = \Bor(G)$ of Borel subgroups of $G$;
\item  \smallskip the line bundle $\mathcal{O}(1)$ by the anti-canonical bundle $L = - \textup{K}_X$ of $X$; 
\item \smallskip the norm $\| \cdot \|$ by the metric $\| \cdot \|_{\mathcal{L}}$ on $L$ induced by the line bundle 
$$ \mathcal{L} = (\det \Omega^1_{\mathcal{X} / k^\circ})^\vee \otimes \alpha^\ast (\det \Lie \mathcal{G})^\vee$$
on the $k^\circ$-scheme of Borel subgroups $\mathcal{X} = \Bor(\mathcal{G})$ of $\mathcal{G}$, where $\alpha$ is the structural morphism of $\mathcal{X}$ and $\Lie \mathcal{G}$ the Lie algebra of $\mathcal{G}$.
\end{itemize}

\smallskip 
Note that, when $G = \GL_n$, these new choices do not correspond to the original ones so that even in this case we get a new (but slightly more complicated) proof.

The construction of the metric $\| \cdot \|_{\mathcal{L}}$ is inspired by the embedding of the Bruhat-Tits building in the flag varieties defined by Berkovich and R\'emy-Thuillier-Werner \cite{RemyThuillierWerner}.

The anti-canonical bundle $L$ of $X$ has a natural structure of $G$-linearised sheaf, that is, $G$ acts linearly on the fibres on $L$ respecting the action on $X$. We can therefore consider the stabiliser $\Stab_{G(k)}(\| \cdot \|_{\mathcal{L}})$ in $G(k)$ of the metric $\| \cdot \|_{\mathcal{L}}$ (see also paragraph \ref{par:Definitions}).

\begin{theorem} \label{Thm:StabiliserMetricRational} Let us suppose that $G$ is semi-simple and quasi-split. Then,
$$ \Stab_{G(k)}(\| \cdot \|_{\mathcal{L}}) = \mathcal{G}(k^\circ). $$
\end{theorem}

\subsection{} 
Theorem \ref{Thm:StabiliserMetricRational} is the critical result that we need to prove Theorem \ref{Thm:HyperspecialSubgroupIntro} when $G$ is quasi-split. It corresponds indeed to step (1) in the proof in the case $\mathcal{G} = \GL_{n, k^\circ}$, whereas step (2) is trivial and step (3) is a standard fact in the theory of reductive $k^\circ$-groups (see Proposition \ref{Prop:TransitiveActionOnBorelVariety}). 

To show Theorem \ref{Thm:StabiliserMetricRational} we reduce the problem to studying the intersection of the stabiliser with the unipotent radical $\rad^u(B)$ of a Borel subgroup $B$ of $G$. Then, identifying $\rad^u(B)$ with the open subset $\Opp(B)$ of Borel subgroups opposite to $B$, it remains to understand the behaviour of the metric $\| \cdot \|_{\mathcal{L}}$ on $\Opp(B, k)$: this boils down to a basic fact in the theory of Schubert varieties over the residue field $\tilde{k}$ (see Proposition \ref{Prop:MetricOfTheEigenvector}).

\subsection{} 
Instead if $G$ is not quasi-split (hence the residue field infinite), then $X(k)$ is empty by definition and the metric $\| \cdot \|_{\mathcal{L}}$ gives no information. To get round this problem we remark that, for every analytic extension\footnote{\textit{i.e.} a complete valued field endowed with an isometric embedding $k \to K$.} $K$ of $k$, $\mathcal{G}(K^\circ)$ is the $K$-holomorphically convex envelope of $\mathcal{G}(k^\circ)$ (see Definition \ref{def:HolomorphicallyConvexEnvelope}). The key point here is that, the residue field being infinite, the $\tilde{k}$-valued points of $\mathcal{G}$ are Zariski-dense in the special fibre of $\mathcal{G}$.

Then, choosing an analytic extension $K$ that splits $G$, we deduce the maximality of $\mathcal{G}(k^\circ)$ from the maximality of $\mathcal{G}(K^\circ)$,  which holds by the quasi-split case. 

\subsection{} 
To conclude let us remark that the construction of the metric can be generalised to any type of parabolic subgroups of $G$. When $\mathcal{G}$ is semi-simple and the type is $k$-rational\footnote{Namely, the corresponding connected component of the variety of parabolic subgroups $\Par(G)$ has a $k$-rational point.} and non-degenerate\footnote{That is, the restriction of a parabolic subgroup of type $t$ to every quasi-simple factor $H$ of $G$ is not the whole $H$ (see \cite[3.1]{RemyThuillierWerner}).}, the stabiliser is still $\mathcal{G}(k^\circ)$. Since these facts are of no use in the present paper we do not treat them.

\subsection{Organisation of the paper} In section \ref{par:Notation} we introduce the notations that we use throughout the paper and we recall some basic facts on reductive groups and Berkovich spaces. In section \ref{sec:ReductionToStabilisers} we show how
Theorem \ref{Thm:HyperspecialSubgroupIntro} follows from Theorem \ref{Thm:StabiliserMetricRational} when $G$ is quasi-split. The proof of Theorem \ref{Thm:HyperspecialSubgroupIntro} for $G$ quasi-split is given in section \ref{sec:ProofOfMainTheoremRational}, based on some preliminary facts established in section \ref{sec:PreliminaryFacts}. Finally in section \ref{sec:InfiniteResidueField} we show how to reduce to the quasi-split case.

\subsection{Acknowledgements} I would like to thank B. R\'emy for the interest he showed on this result, the accurate reading of a first draft of this paper and his commentaries. I warmly thank A. Thuillier for interesting discussions, M. Brion for pointing out to me references for Lemma \ref{Lemma:NonVanishingLocusOfTheEigenvector} and B. Conrad for giving me an alternative argument for 
Proposition \ref{Prop:TransitiveActionOnBorelVariety}. I also thank G. Ancona for the valuable suggestions about the presentation. 

I would like to thank the referee for its careful reading, his advices on the presentation and the suggestion to prove Theorem \ref{Thm:StabiliserMetricRational} when the group is quasi-split which permitted me to give a more elementary treatment.\footnote{In a previous version of the paper I proved Theorem \ref{Thm:StabiliserMetricRational} only when the residue field was finite and used Berkovich geometry when the residue field was infinite.} 

\section{Notations, reminders and definitions} \label{par:Notation} 

\subsection{Notations and conventions} Let us list some notations that we use throughout the paper:

\begin{itemize}
\item $k$ is a non-archimedean field, $k^\circ$ its ring of integers and $\tilde{k}$ its residue field; 
\item \smallskip $\mathcal{G}$ is a reductive $k^\circ$-group; 
\item \smallskip $\mathcal{X}$ is the $k^\circ$-scheme of Borel subgroups $\Bor(\mathcal{G})$, that is, the $k^\circ$-scheme representing the functor that associates to a $k^\circ$-scheme $S$ the set of Borel subgroups of the reductive $S$-group $\mathcal{G} \times_{k^\circ} S$ (\textit{cf.} \cite[XXII, Corollaire 5.8.3]{SGA3}); 
\item \smallskip $\mathcal{L}$ is the invertible sheaf $(\det \Omega_{\mathcal{X} / k^\circ})^\vee \otimes  \alpha^\ast (\det \Lie \mathcal{G})^\vee$ on $\mathcal{X}$ (it is a line bundle because $\mathcal{X}$ is smooth by \textit{loc.cit.}), where $\alpha$  is the structural morphism of $\mathcal{X}$ over $\Spec k^\circ$; 
\item \smallskip $G$, $X$, $L$  are respectively the generic fibre of $\mathcal{G}$, $\mathcal{X}$, $\mathcal{L}$  and by $\tilde{G}$, $\tilde{X}$, $\tilde{L}$ their special fibre; 
\item \smallskip $\| \cdot \|_{\mathcal{L}}$ is the metric on $L$ associated to $\mathcal{L}$, that we consider as a continuous function $\| \cdot \|_\mathcal{L} \colon \V(L, k) \to \R_+$ (see definition in paragraph \ref{par:Definitions}). 
\item \smallskip for every Borel subgroup $\mathcal{B}$ we denote by $\Opp(\mathcal{B})$ the $k^\circ$-scheme of Borel subgroups opposite to $\mathcal{B}$, that is, the $k^\circ$-scheme representing the functor that associates to a $k^\circ$-scheme $S$ the set of Borel subgroups of  $\mathcal{G}_S \df \mathcal{G} \times_{k^\circ} S$ such that the intersection with $\mathcal{B} \times_{k^\circ} S$ is a maximal $S$-torus of $\mathcal{G}_S$.  A similar notation is also used for Borel subgroups of $G$ and $\tilde{G}$ (\textit{cf.} \cite[XXII, Proposition 5.9.3 (ii)]{SGA3}). 
\item \smallskip  In this paper we refer to \cite[Corollaire 1.11]{SGA3ExpXI} as ``Hensel's Lemma''.
\end{itemize}

\subsection{Reminders} 

\begin{itemize}
\item For a reductive group over a general base, the notion of quasi-split is fairly involved \cite[XXIV, 3.9]{SGA3}. Nonetheless, thanks to \cite[XXIV, Proposition 3.9.1]{SGA3}, the $k^\circ$-reductive group $\mathcal{G}$ is quasi-split if and only if $G$ is. 

\item \smallskip The $k^\circ$-scheme $\mathcal{X}$ is projective and smooth (see \cite[Theorem 5.2.11]{ConradSGA3} or \cite[XXII, 5.8.3 (i)]{SGA3}) and the invertible sheaf $\mathcal{L}$ is ample. Indeed, $\mathcal{L}$ can also be constructed as follows: if $\mathcal{U} \to \mathcal{X}$ is the universal Borel subgroup and $\Lie \mathcal{U}$ is the Lie algebra of $\mathcal{U}$, then $\mathcal{L}$ is the dual of $\det \Lie \mathcal{U}$ \cite[Theorem 2.3.6 and Remark 2.3.7]{ConradSGA3}. 

This construction also shows that the adjoint action of $\mathcal{G}$ induces a natural equivariant action of $\mathcal{G}$ on $\mathcal{L}$ \cite[I, D\'efinition 6.5.1]{SGA3}. The equivariant action on $\mathcal{L}$ induces for all integer $n$ a linear action of $\mathcal{G}$ on the global sections $H^0(\mathcal{X}, \mathcal{L}^{\otimes n})$ \cite[I, Lemme 6.6.1]{SGA3}. We always consider these actions as tacitly understood. 

\item \smallskip For a Borel subgroup $\mathcal{B}$ of $\mathcal{G}$ the scheme $\Opp(\mathcal{B})$ of Borel subgroups of $\mathcal{G}$ opposite to $\mathcal{B}$ is an open affine subscheme of $\mathcal{X} = \Bor(\mathcal{G})$ \cite[XXVI, Corollaire 4.3.4 and Corollaire 4.3.5]{SGA3}.

\item \smallskip The total space of $L$ is the $k$-scheme $\V(L)$ representing the functor that associates to a $k$-scheme $S$ the set of couples $(x, s)$ made of a $S$-valued point $x \colon S \to X$ and a section $s \in H^0(S, x^\ast L)$ \cite[1.7.10]{EGA2}.\footnote{In \textit{loc.cit.} the $k$-scheme $\V(L)$ is denoted $\V(L^\vee)$.}

\end{itemize}

\subsection{Definitions} \label{par:Definitions}

\begin{itemize}
\item A subset $S \subset G(k)$ is said to be \textit{bounded} if there exists a closed embedding $G \subset \A^{n}_k$  such that $S$ in contained in $\A^n(k^\circ)$ (this generalises \cite[1.1, Definition 2]{NeronModels} when $k$ is not discretely valued). 

\item  \smallskip A \textit{metric} on $L$ is a function $\| \cdot \| \colon \V(L, k) \to \R_+$, $(x, s) \mapsto \| s\|(x)$ verifying the following properties for all $k$-points $(x, s)$ of $\V(L)$:
\begin{itemize}
\item \smallskip $\| s\|(x) = 0$ if and only if $s = 0$;
\item \smallskip $\| \lambda s \| (x) = |\lambda| \| s\|(x)$ for all $\lambda \in k$.
\end{itemize} \smallskip

\item \smallskip The metric $\| \cdot \|_{\mathcal{L}}$ is defined as follows. A $k$-point of $\V(L)$ corresponds to the data of a point $x \in X(k)$ and a section $s \in x^\ast L$. By the valuative criterion of properness, the point $x$ lifts to a unique morphism of $k^\circ$-schemes $\epsilon_x \colon \Spec k^\circ \to \mathcal{X}$ and the $k^\circ$-module $\epsilon_x^\ast \mathcal{L}$ is free of rank $1$ (thus it is a lattice the $K$-line $x^\ast L$). Pick a generator $s_0$ of the $k^\circ$-module $\epsilon_x^\ast \mathcal{L}$ and set, for all $s = \lambda s_0$ with $\lambda \in k$,
$$ \| s \|_{\mathcal{L}}(x) \df |\lambda|.$$
The real number $ \| s \|_{\mathcal{L}}(x)$ does not depend on the chosen generator $s_0$, so this gives a well-defined function $\| \cdot \|_\mathcal{L} \colon \V(L, k) \to \R_+$, 
$$ \| (x, s) \|_{\mathcal{L}} \df \| s\|_\mathcal{L}(x). $$

It is easily seen that $\| \cdot \|_\mathcal{L}$ is continuous on $\V(L, k)$ and bounded on bounded subsets. Similarly, for every integer $n$, one constructs the metric $\| \cdot \|_{\mathcal{L}^{\otimes n}}$ on $L^{\otimes n}$ associated to $\mathcal{L}^{\otimes n}$. 

\item \smallskip The group $G(k)$ acts on the set of metrics on $L$. Indeed, given a metric $\| \cdot \|$ and $g \in G(k)$, the function $(x, s) \mapsto \| g^{-1} \cdot (x, s)\|$ is again a metric (because $G(k)$ acts linearly on the fibres of $L$).

\smallskip
We denote by $\Stab_{G(k)}(\| \cdot \|_{\mathcal{L}})$ the stabiliser of the metric $\| \cdot \|_{\mathcal{L}}$ with respect to this action. More explicitly, $\Stab_{G(k)}(\| \cdot \|_{\mathcal{L}})$ is the set of points $g \in G(k)$ such that, for all $k$-points $(x, s)$ of $\V(L)$, we have
$$ \| g \cdot s \|_\mathcal{L}(g \cdot x) = \| s \|_{\mathcal{L}}(x).$$
\end{itemize}

\section{Proof of Theorem \ref{Thm:HyperspecialSubgroupIntro} in the quasi-split case} \label{sec:ReductionToStabilisers}

In this section we admit temporarily Theorem \ref{Thm:StabiliserMetricRational} and we prove the following:

\begin{theorem} \label{Thm:HyperspecialSubgroupQuasiSplit} Let us suppose $\mathcal{G}$ quasi-split.  Then, $\mathcal{G}(k^\circ)$ is a maximal bounded subgroup of $G(k)$.
\end{theorem}

\begin{proof}[{Proof of Theorem \ref{Thm:HyperspecialSubgroupQuasiSplit}}] We may assume that $G$ is semi-simple. Indeed, if it is not the case, we consider the derived group $\mathcal{D}$ of $\mathcal{G}$ (which is a semi-simple $k^\circ$-group scheme \cite[XXII, Th\'eor\`eme 6.2.1 (iv)]{SGA3}) and the identity component of the center $\mathcal{Z}$ of $\mathcal{G}$ (which is a $k^\circ$-torus). The map $\pi \colon \mathcal{D} \times_{k^\circ} \mathcal{Z} \to \mathcal{G}$ given by multiplication  is an isogeny \cite[XXII, Proposition 6.2.4]{SGA3}. If $H$ is a bounded subgroup of $G(k)$ containing $\mathcal{G}(k^\circ)$, then the subgroup $\pi^{-1}(H)$ contains $\mathcal{D}(k^\circ) \times \mathcal{Z}(k^\circ)$ and is bounded because $\pi$ is a finite morphism\footnote{If $V, W$ are $k$-schemes of finite type and $f \colon V \to W$ is a finite morphism, then the inverse image of a bounded subset of $W(k)$ is bounded. Since finite morphisms are projective, in order to prove this statement, one is immediately led back to prove it when $V = \P^n \times_k W$ and $f$ is the projection on the second factor. This latter statement is clear because $\P^n(k)$ is bounded (the proof given in \cite[1.1, Proposition 6]{NeronModels} when $k$ is discretely valued generalises without problems to the non-discretely valued case).}. Since $\mathcal{Z}(k^\circ)$ is the maximal bounded subgroup of $Z(k)$ (Proposition \ref{Prop:MaxCompactSubgroupTorus}), we are left with proving that $\mathcal{D}(k^\circ)$ is a maximal bounded subgroup of $D(k)$. 

Let us henceforth suppose that $G$ is semi-simple. Let $H$ be a bounded subgroup containing $\mathcal{G}(k^\circ)$ and let us consider the metric $\| \cdot \|_H$ on $L_{ \rvert X(k)}$ defined, for every point $x \in X(k)$ and every section $s \in x^\ast L$, by
$$ \| s\|_H(x) \df \sup_{h \in H} \| h \cdot s\|_{\mathcal{L}} (h \cdot x).$$
Note that $\| \cdot \|_H$ takes real values because $H$ is bounded and $\| \cdot \|_{\mathcal{L}}$ is continuous and bounded. The ratio of the metrics $\| \cdot \|_H$ and $\| \cdot \|_{\mathcal{L}}$ defines a function
$$ \phi = \frac{\| \cdot \|_H}{\| \cdot \|_{\mathcal{L}}} \colon X(k) \too \R_+,$$
which is invariant under the action of $\mathcal{G}(k^\circ)$. Since $\mathcal{G}(k^\circ)$ acts transitively on $X(k)$ (Proposition \ref{Prop:TransitiveActionOnBorelVariety} (\ref{Prop:TransitiveActionOfTheStabiliser})), the function $\phi$ must be constant. Thus $H$ is contained in  $\Stab_{G(k)}(\| \cdot \|_{\mathcal{L}})$  and, according to Theorem \ref{Thm:StabiliserMetricRational}, we conclude. 
\end{proof}

\section{Some preliminary facts} \label{sec:PreliminaryFacts}

In this section we collect some facts that will be used during the proof of Theorem \ref{Thm:StabiliserMetricRational}. Some of them are standard facts but we included their proof for the sake of completeness.

\subsection{On the scheme of Borel subgroups} A perfect field $F$ is said to be of cohomological dimension $\le 1$ if every homogeneous space under a connected linear algebraic group has a $F$-rational point \cite[\S 2.2, III, Th\'eor\`eme 1 and \S 2.3, Corollaire 1]{CohomologieGaloisienne}. The only examples of fields of cohomological dimension $\le 1$ we are interested in are finite fields (``Lang's theorem'' \cite[Corollary 16.5 (i)]{Borel}).

\begin{proposition} \label{Prop:GroupIsQuasiSplit} Let us suppose that the residue field $\tilde{k}$ is perfect of cohomological dimension $\le 1$ and let $\mathcal{G}$ be a $k^\circ$-reductive group. Then, its generic fibre $G$ is quasi-split.
\end{proposition}

\begin{proof} The special fibre $\tilde{X}$ of $\mathcal{X}$ is a homogeneous space under the action of the connected group $\tilde{G}$. Therefore, by definition of field of cohomological dimension $\le 1$, it admits a $\tilde{k}$-rational point. Thanks to the smoothness of $\mathcal{X}$ and Hensel's lemma, such a rational point can be lifted to a $k^\circ$-valued point of $\mathcal{X}$, that is, to a Borel subgroup $\mathcal{B}$ of $\mathcal{G}$. The generic fibre of $\mathcal{B}$ does the job.
\end{proof} 

\begin{proposition} \label{Prop:TransitiveActionOnBorelVariety} Let us suppose $\tilde{k}$ arbitrary and $G$ quasi-split. Then,
\begin{enumerate}
\item \label{Prop:MaxTorusInABorel} every Borel subgroup $\mathcal{B}$ of $\mathcal{G}$ contains a maximal torus of $\mathcal{G}$;
\item \label{Prop:TransitiveActionOfTheStabiliser} $\mathcal{G}(k^\circ)$ acts transitively on $X(k)$;
\item \label{Prop:IwasawaDecomposition} (Iwasawa decomposition) for every Borel subgroup $B$ of $G$, we have $$G(k) = \mathcal{G}(k^\circ) \cdot B(k). $$
\end{enumerate}
\end{proposition}

\begin{proof} (1) \cite[XII, Corollaire 5.9.7]{SGA3}. (2) We can apply \cite[XXVI, Corollaire 5.2]{SGA3}. Indeed, if $\mathcal{B}$ is a Borel subgroup of $\mathcal{G}$, by (1) it contains a maximal torus $\mathcal{T}$  and we can consider the Borel subgroup $\mathcal{B}'$ opposite to $\mathcal{B}$ with respect to $\mathcal{T}$ \cite[XXII, Proposition 5.9.2]{SGA3}. 

(3) Since $\mathcal{X}$ is proper, the valuative criterion of properness entails the equality $\mathcal{X}(k^\circ) = X(k)$, which, according to (2), gives
$$ \mathcal{G}(k^\circ) / \mathcal{B}(k^\circ) = G(k) / B(k).$$
The result follows immediately.
\end{proof}

\subsection{Size of global sections of the anti-canonical bundle} Let us start by recalling a basic fact in the theory of Schubert varieties over a field.

\begin{lemma} \label{Lemma:NonVanishingLocusOfTheEigenvector} Let $F$ be a field. Let $H$ be a quasi-split reductive $F$-group and $P$ a Borel subgroup of $H$. Let $Y$ be the variety of Borel subgroups of $H$ and $M$ the anti-canonical bundle of $Y$. Then,
\begin{enumerate}
\item there exists a unique (up to scalar factor) non-zero eigenvector in $H^0(Y, M)$ for $P$;
\item the locus where such an eigenvector does not vanish is the open subset $\Opp(P) \subset Y$ made of Borel subgroups opposite to $P$.
\end{enumerate}
\end{lemma}

\begin{proof} When $H$ is split, $Y$ is the Schubert variety associated to the maximal element $w_0$ of the Weyl group of $H$ (with respect to the Bruhat order) and $\Opp(P)$ is the corresponding Bruhat cell -- see, for instance, \cite[Proposition 1.4.5]{Brion}, \cite[\S 8.5.7]{Springer} or \cite[\S 4]{Kempf} for a thorough discussion of these aspects. The quasi-split case follows by Galois descent.
\end{proof}

Let us go back to the general notation introduced in paragraph \ref{par:Notation}. 

\begin{proposition} \label{Prop:MetricOfTheEigenvector} Let us suppose $G$ quasi-split and let $B$ be a Borel subgroup. Let $s \in H^0(\mathcal{X}, \mathcal{L})$ be an eigenvector for $B$ such that its reduction $\tilde{s}$ is non-zero. Then,
$$ \{ x \in X(k) : \| s\|_{\mathcal{L}} (x) = 1\} = \Opp(\mathcal{B}, k^\circ),$$
where $\mathcal{B}$ is the Borel subgroup of $\mathcal{G}$ lifting $B$ and $\Opp(\mathcal{B})$ is the open subset of $\Bor(\mathcal{G})$ made of Borel subgroups opposite to $\mathcal{B}$.
\end{proposition}

\begin{remark}
This statement is a ``coordinate-free'' analogue of \cite[Proposition 2.18 (i)]{RemyThuillierWerner} (in the sense that we do not need to consider a maximal split $k$-torus of $G$ and the corresponding roots).
\end{remark}

\begin{proof} Let $x \in X(k)$. First of all, applying Lemma \ref{Lemma:NonVanishingLocusOfTheEigenvector} with $F = k$ and $H = G$, let us remark that we have $\| s\|_\mathcal{L}(x) \neq 0$ exactly when $x \in \Opp(B, k)$. Furthermore, the equality $\| s\|_\mathcal{L}(x) = 1$ is equivalent to say that the reduction $\tilde{s} \in H^0(\tilde{X}, \tilde{L})$ of $s$ does not vanish at the reduction $\tilde{x} \in \tilde{X}(\tilde{k})$ of $x$.

If $\tilde{B}$ denotes special fibre of $\mathcal{B}$, then $\tilde{s}$ is a non-zero eigenvector for $\tilde{B}$. Therefore, applying again Lemma \ref{Lemma:NonVanishingLocusOfTheEigenvector} to $F = \tilde{k}$ and $H = \tilde{G}$, we obtain that $\tilde{s}$ does not vanishes precisely on the open subset $\Opp(\tilde{B})$ of $\tilde{X}$ made of Borel subgroups of $\tilde{G}$ opposite to $\tilde{B}$. 
Let $\mathcal{B}_x$ the Borel subgroup of $\mathcal{G}$ associated to $x$. Summing up we have:
\begin{eqnarray*}
\| s\|_{\mathcal{L}}(x) \neq 0  & \Longleftrightarrow & \textup{the generic fibre of $\mathcal{B}_x$ is opposite to $B$}, \\
\| s\|_{\mathcal{L}}(x) = 1  & \Longleftrightarrow  & \hspace{-5pt}
\begin{tabular}{l}
\textup{the generic and the special fibre of $\mathcal{B}_x$ are}\\ 
\textup{respectively opposite to $B$ and $\tilde{B}$.}
\end{tabular}
\end{eqnarray*}
In other words, $\| s\|_{\mathcal{L}}(x) = 1$ if and only if $x \in \Opp(\mathcal{B}, k^\circ)$.
\end{proof}

\subsection{Compact subgroups of tori} Given a torus $T$ over a non-archimedean field $k$ (complete and non-trivially valued), the set of its $k$-rational points contains a unique maximal bounded subgroup $U_T$.\footnote{Indeed, if $T \iso \G_{m, k}^r$ is split one takes $U_T = \Gm^r(k^\circ)$. If $T$ is not split, let $k'$ be a finite separable extension splitting $T$ and let $T' = T \times_k k'$. Then, $U_T = U_{T'} \cap T(k)$. }

It is not true in general that $U_T$ is the group of $k^\circ$-valued points of a $k^\circ$-torus $\mathcal{T}$. When $k$ is discretely valued, $U_T$ coincides with the set of $k^\circ$-valued points of the identity component $\mathcal{T}$ of the N\'eron model of $T$ \cite[4.4.12]{BruhatTitsII} but, if the splitting extension of $K$ is ramified, then the special fibre of $\mathcal{T}$ may not be a torus \cite[4.4.13]{BruhatTitsII}. Anyway, this is true if $T$ is already the generic fibre of $k^\circ$-torus : 

\begin{proposition} \label{Prop:MaxCompactSubgroupTorus} Let $\mathcal{T}$ be a $k^\circ$-torus and $T$ its generic fibre. Then,  $\mathcal{T}(k^\circ)$ is the unique maximal bounded subgroup of $T(k)$.
\end{proposition}

\begin{proof} If $\mathcal{T} \iso \G_{m, k^\circ}^r$ is split, then $\Gm^r(k^\circ)$ is the unique maximal bounded subgroup of $\Gm^r(k)$. In general there exists a finite unramified extension $K$ of $k$ such that $\mathcal{T}_K \df \mathcal{T} \times_{k^\circ} K^\circ$ is split and, by the split case, $\mathcal{T}(K^\circ)$ is the unique maximal bounded subgroup of $T(K)$. It follows that $\mathcal{T}(K^\circ) \cap T(k) = \mathcal{T}(k^\circ)$ is the unique maximal bounded subgroup of $T(k)$.
\end{proof}

\subsection{Boundedness of the stabiliser} In this section we establish that the stabiliser $\Stab_{G(k)}(\| \cdot \|_{\mathcal{L}})$ is a bounded subset of $G(k)$. Let us begin with two results that we need in the proof.

\begin{lemma} \label{Prop:FinitenessRepresentation} Let $n \ge 1$ be such that $L^{\otimes n}$ is very ample. If $G$ is semi-simple, then the natural representation $ \rho \colon G \to \GL(H^0(X, L^{\otimes n}))$ is finite as a morphism of $k$-schemes.
\end{lemma}

\begin{proof} We may assume that $k$ is algebraically closed. We prove that $\Ker \rho$ is finite, which clearly implies the statement.

Since $X$ embeds $G$-equivariantly in $\P(H^0(X, L^{\otimes n})^\vee)$, then $\Ker \rho$ is contained in the stabiliser of every point of $X$. That is, $\ker \rho$ is contained in the intersection of all Borel subgroups. In other words, the identity component of $\Ker \rho$ is the radical of $G$, which is trivial since $G$ is semi-simple \cite[\S 11.21]{Borel}. \end{proof}

\begin{lemma} \label{Prop:BoundednessStabiliserNorm} Let $V$ be a finite dimensional $k$-vector space and let $\| \cdot \|$ be a norm on $V$. Then the following subgroup of $\GL(V, k)$,
$$ \Stab_{\GL(V, k)}(\| \cdot \|) \df \{ g \in \GL(V, k) : \| g \cdot v \| = \| v \| \textup{ for all } v \in V \},$$
is bounded.
\end{lemma}

\begin{proof} Let us see $\GL(V)$ as a closed subscheme of the affine scheme $\End(V) \times_k \End(V)$ through the closed embedding $g \mapsto (g, g^{-1})$. If we consider the subset 
$$ E = \{ \phi \in \End(V, k) : \| \phi(v) \| \le \| v\| \textup{ for all } v \in V  \},$$
then we have $\Stab_{\GL(V, k)}(\| \cdot \|) = (E \times E) \cap \GL(V, k)$. Therefore it suffices to show that the subset $E$ is bounded. Let $\mathcal{V}_1$, $\mathcal{V}_2$ be $k^\circ$-lattices of $V$ such that the associated norms on $V$ satisfy, for all $v \in V$,
$$ \| v \|_1 \le \| v \| \le \| v \|_2,$$
(they exist because the norms on $V$ are all equivalent). It follows, through the canonical isomorphism $\End(V) = \Hom_k(V, k) \otimes_k V$, that $E$ is a subset of 
$$\Hom_{k^\circ}(\mathcal{V}_2, k^\circ) \otimes_{k^\circ} \mathcal{V}_1. $$ In particular $E$ is bounded by definition.
\end{proof}

\begin{proposition}\label{Lemma:StabiliserIsCompactRational} If $\mathcal{G}$ is semi-simple and quasi-split, then $\Stab_{G(k)}(\| \cdot \|_{\mathcal{L}})$ is bounded. 
\end{proposition}

\begin{proof} Let $n \ge 1$ be an integer such that $L^{\otimes n}$ is very ample, let $V \df H^0(X, L^{\otimes n})$ and let $\rho \colon G \to \GL(V)$ be the representation induced by the equivariant action of $G$ on $L^{\otimes n}$. According to Lemma \ref{Prop:FinitenessRepresentation}, $\rho$ is a finite morphism.

For every global section $s \in V$ let us set
$$ \| s \|_{\sup} \df \sup_{x \in X(k)} \| s \|_{\mathcal{L}^{\otimes n}}(x). $$
Remark that $\| \cdot \|_{\sup}$ is a norm on $V$ because $X(k)$ is non-empty and thus, by the Zariski-density of $G(k)$ in $G$, Zariski-dense in $X$ \cite[Theorem 1.1]{ConradUnirationality}. 

The subgroup $S = \Stab_{G(k)}(\| \cdot \|_{\mathcal{L}})$ fixes the norm $\| \cdot \|_{\sup}$, therefore its image in $\GL(V, k)$ through $\rho$ is bounded (Lemma \ref{Prop:BoundednessStabiliserNorm}). Since $\rho$ is a finite morphism, $S$ must be bounded too.
\end{proof}

\section{Proof of Theorem \ref{Thm:StabiliserMetricRational}} \label{sec:ProofOfMainTheoremRational}

In this section we prove Theorem  \ref{Thm:StabiliserMetricRational} and therefore we suppose that the group $G$ is semi-simple and quasi-split. 

\subsection{}  
In order to prove Theorem \ref{Thm:StabiliserMetricRational}, we start by noticing that the metric $\| \cdot \|_{\mathcal{L}}$ is invariant under $\mathcal{G}(k^\circ)$. Indeed, let $x \in X(k)$, $g \in \mathcal{G}(k^\circ)$ and let us denote by $\epsilon_{x}$, $\epsilon_{g \cdot x}$ the unique $k^\circ$-valued points of $\mathcal{X}$ that lift, by valuative criterion of properness, respectively the points $x$ and $g \cdot x$. Since $\mathcal{G}$ acts equivariantly on $\mathcal{L}$, the multiplication by $g$ induces an isomorphism of $k^\circ$-modules $$\epsilon_x^\ast \mathcal{L} \stackrel{\sim}{\too} \epsilon_{ g \cdot x}^\ast \mathcal{L},$$ 
extending the isomorphism of $k$-vector spaces $ x^\ast L \to (g \cdot x)^\ast L$. For a section $s \in x^\ast L$ let us write $g \cdot s$ its image in $(g \cdot x)^\ast L$. Since the isomorphism is defined at the level of $k^\circ$-modules, if $s_0$ is a generator of the $k^\circ$-module $\epsilon_x^\ast \mathcal{L}$, then $g \cdot s_0$ generates the $k^\circ$-module $\epsilon_{ g \cdot x}^\ast \mathcal{L}$. In particular, for every section $s \in x^\ast L$, we have
$$ \| g \cdot s \|_{\mathcal{L}}(g \cdot x) = \| s \|_\mathcal{L}(x).$$

We are thus left with proving the inclusion
\begin{equation} \label{eq:InclusionInTheStabiliserRational}  \Stab_{G(k)}(\| \cdot \|_{\mathcal{L}}) \subset \mathcal{G}(k^\circ). \end{equation}

Since $G$ is supposed to be quasi-split, it contains a Borel subgroup $B$ and by the Iwasawa decomposition (Proposition \ref{Prop:TransitiveActionOnBorelVariety} (\ref{Prop:IwasawaDecomposition})), we have
$$ G(k) = \mathcal{G}(k^\circ) \cdot B(k).$$
Therefore, in order to prove the inclusion \eqref{eq:InclusionInTheStabiliserRational}, it suffices to prove the following : 

\begin{lemma} \label{lem:ReductionToABorelRational}With the notations just introduced, let $\mathcal{B}$ be the unique Borel subgroup of $\mathcal{G}$ lifting $B$. Then, we have
$$  \Stab_{G(k)}(\| \cdot \|_{\mathcal{L}}) \cap B(k) = \mathcal{B}(k^\circ). $$
\end{lemma}

\subsection{} \label{par:ReductionToRadical} 
Let us prove Lemma \ref{lem:ReductionToABorelRational}. Let us simplify the notation by writing $S$ instead of $\Stab_{G(k)}(\| \cdot \|_{\mathcal{L}})$. Let $\mathcal{T}$ be a maximal $k^\circ$-torus of $\mathcal{B}$ (it exists by Proposition \ref{Prop:TransitiveActionOnBorelVariety} (\ref{Prop:MaxTorusInABorel})) and $\rad^u(\mathcal{B})$ be the unipotent radical of $\mathcal{B}$. Let $T$ and $\rad^u(B)$ be their generic fibres.

The inclusion $\mathcal{T} \subset \mathcal{B}$ induces an isomorphism $\mathcal{T} \iso \mathcal{B} / \rad^u(\mathcal{B})$ \cite[XXVI, Proposition 1.6]{SGA3}. Thanks to this identification, let us write $\pi \colon \mathcal{B} \to \mathcal{T}$ the quotient map. 

\begin{claim} We have the following equalities:
$$S \cap T(k) = \mathcal{T}(k^\circ) , \quad \pi(S \cap B(k)) = \mathcal{T}(k^\circ).$$
\end{claim}

\begin{proof}[Proof of the Claim] Since the stabiliser $S$ is a bounded subgroup (Proposition \ref{Lemma:StabiliserIsCompactRational}), the subgroups $S \cap T(k)$, $\pi(S \cap B(k))$ of $T(k)$ are bounded too. Therefore they must be contained in $\mathcal{T}(k^\circ)$ because the latter is the unique maximal bounded subgroup of $T(k)$ (Proposition \ref{Prop:MaxCompactSubgroupTorus}).

On the other hand $S$ contains $\mathcal{G}(k^\circ)$ by hypothesis and thus it contains $\mathcal{T}(k^\circ)$. So both $S \cap T(k)$ and $\pi(S \cap B(k))$ contain $\mathcal{T}(k^\circ)$, whence the claim.
\end{proof}

Since $B(k)$ is the semi-direct product of $\rad^u(B, k)$ and $T(k)$, in order to conclude the proof of Lemma \ref{lem:ReductionToABorelRational}, it is sufficient to prove the following:

\begin{lemma} \label{lem:ReductionToRadicalRational} With the notations introduced above, we have
$$  \Stab_{G(k)}(\| \cdot \|_{\mathcal{L}}) \cap \rad^u(B, k) = \rad^u(\mathcal{B}, k^\circ). $$
\end{lemma}

\subsection{} \label{par:EndOfTheProof} 
Let us prove Lemma \ref{lem:ReductionToRadicalRational}. Let $s \in H^0(\mathcal{X}, \mathcal{L})$ be an eigenvector for $B$ whose reduction $\tilde{s}$ is non-zero. Then, by Proposition \ref{Prop:MetricOfTheEigenvector} we have
$$ \Opp(\mathcal{B}, k^\circ) = \{ x \in X(k) : \| s\|_{\mathcal{L}}(x) = 1 \}.$$
Since the subgroup $S \cap \rad^u(B, k)$ fixes the metric, $\Opp(\mathcal{B}, k^\circ)$ is stable under the action of $S \cap \rad^u(B, k)$.

On the other hand, we can identify in a $\mathcal{B}$-equivariant way $\Opp(\mathcal{B})$ with the unipotent radical of $\mathcal{B}$. To do this, let $\mathcal{B}^\op$ be the Borel subgroup of $\mathcal{G}$ opposite to $\mathcal{B}$ relatively to $\mathcal{T}$. Then the map $\rad^u(\mathcal{B}) \to \Opp(\mathcal{B})$ defined by  $b \mapsto b \mathcal{B}^\op b^{-1}$ is an isomorphism \cite[XXVI, Corollaire 4.3.5]{SGA3}. 

Through this identification, the action of $\rad^u(B)$ on $\Opp(B)$ becomes the action of $\rad^u(B)$ on itself by left multiplication. Moreover, saying that $\Opp(\mathcal{B}, k^\circ)$ is stable under the action of $S \cap \rad^u(B)$ translates into the the fact that the unipotent radical $\rad^u(\mathcal{B}, k^\circ)$ is stable under the left multiplication by $S \cap \rad^u(B)$. This obviously implies that $S \cap \rad^u(B)$ is contained in $\rad^u(\mathcal{B}, k^\circ)$, which concludes the proof of Lemma \ref{lem:ReductionToRadicalRational}, thus of Lemma \ref{lem:ReductionToABorelRational} and Theorem \ref{Thm:StabiliserMetricRational}.  \qed

\section{Reduction to the quasi-split case} \label{sec:InfiniteResidueField} In this section we deduce Theorem \ref{Thm:HyperspecialSubgroupIntro} when $G$ is not quasi-split from Theorem \ref{Thm:HyperspecialSubgroupQuasiSplit}. The reduction to the quasi-split case makes an essential use of the concept of holomorphically convex envelope, that we pass in review in the first paragraph.

\subsection{Holomorphically convex envelopes} We briefly discuss holomorphically convex envelopes. The naive point of view we opt for, far from being well-suited to study holomorphically convex spaces,  will suffice to draw the result that we are interested in (\textit{cf}. Proposition \ref{Prop:HolomorphicallyConvexEnvelopeRationalPoints}).

\begin{definition} \label{def:HolomorphicallyConvexEnvelope}Let $V$ be a affine $k$-scheme of finite type, $S \subset V(k)$ a bounded subset and $K$ be an analytic extension of $k$. Let $K[V]$ be the $K$-algebra of regular functions on $V \times_k K$. Then, for every $f \in K[V]$, let us set $$ \| f \|_S \df \sup_{s \in S} |f(s)|. $$
The \textit{$K$-holomorphically convex envelope} of $S$ is the subset 
$$ \hat{S}_K \df \{ x \in V(K) : |f(x)| \le \| f\|_S \textup{ for all } f \in K[V]\}. $$
\end{definition}

\begin{proposition} \label{prop:HolomorphicallyConvexEnvelopesClosedImmersions} Let $f \colon V \to W$ be a closed immersion between affine $k$-schemes of finite type. Let $S \subset V(k)$ be a bounded subset and $K$ an analytic extension of $k$. Then,
$$ f(\hat{S}_{K}) = \widehat{f(S)}_{K}.$$
\end{proposition}

The proof is left to reader as a direct consequence of the definitions.

\begin{proposition} \label{Prop:HolomorphicallyConvexEnvelopeRationalPoints} Let $\mathcal{H}$ be a smooth affine group $k^\circ$-scheme with connected geometric fibres. Let us suppose that its special fibre $\tilde{H}$ is unirational and that the residue field $\tilde{k}$ is infinite. 

Then, for every analytic extension $K$ of $k$, the $K$-holomorphically convex envelope of $\mathcal{H}(k^\circ)$ is $\mathcal{H}(K^\circ)$. 
\end{proposition}

We will use the previous Proposition only when $\mathcal{H}$ is a reductive $k^\circ$-group: over a field reductive groups  are indeed unirational varieties \cite[Theorem 1.1]{ConradUnirationality} so that the hypotheses are fulfilled.

In order to prove Proposition \ref{Prop:HolomorphicallyConvexEnvelopeRationalPoints}, let $H$ be the generic fibre of $\mathcal{H}$ and let $K[H]$, $K^\circ[\mathcal{H}]$ be respectively the $k$-algebra of regular functions of $H \times_k K$ and the $k^\circ$-algebra of regular functions on $\mathcal{H} \times_{k^\circ} K^\circ$. For every $f \in K[H]$ let us set 
$$ \| f \|_{K^\circ[\mathcal{H}]} \df \inf \{ |\lambda| : f / \lambda \in K^\circ[\mathcal{H}], \lambda \in K^\times\}.$$ 
The function $ \| \cdot \|_{K^\circ[\mathcal{H}]}$ is a semi-norm on the $K$-algebra $K[H]$ and it takes values in $|K|$. This very last property is crucial for us and it is trivial if the valuation of $K$ is discrete, while, when the valuation is dense, it is known to experts in non-archimedean geometry. A proof of this is given in Appendix \ref{sec:SeminormsAndIntegralModels} (\textit{cf.} Proposition \ref{prop:RationalityOfSeminorms}) as I cannot point out a suitable reference.  Coming back to the proof of Proposition \ref{Prop:HolomorphicallyConvexEnvelopeRationalPoints}, let us remark that we have
$$ \mathcal{H}(K^\circ) = \{ h \in H(K) : |f(h)| \le \| f\|_{K^\circ[\mathcal{H}]} \textup{ for all } f \in K[H]\},$$
so that it suffices to prove the following:

\begin{lemma} For every $f \in K[H]$ we have
$$ \| f \|_{K^\circ[\mathcal{H}]} = \| f \|_{\mathcal{H}(k^\circ)} \df \sup_{h \in \mathcal{H}(k^\circ)} |f(h)|.$$
\end{lemma}

\begin{proof}[Proof of the Lemma] Since the norm $\| \cdot \|_{K^\circ[\mathcal{H}]}$ takes values in $|K|$, we may assume $\| f \|_{K^\circ[\mathcal{H}]} = 1$. With this hypothesis for all points $h \in \mathcal{H}(k^\circ)$ we have $|f(h)| \le 1$, thus proving the lemma amounts to find $h \mathcal{H}(k^\circ)$ such that $|f(h)| = 1$. 

Let $\tilde{H}$ be the special fibre of $\mathcal{H}$ and let $\tilde{K}[\tilde{H}]$ be the $\tilde{K}$-algebra of regular functions on $\tilde{H}_{\tilde{K}} \df \tilde{H} \times_{\tilde{k}} \tilde{K}$. With this notation the function $f$ belongs to $K^\circ[\mathcal{H}]$ and its reduction $\tilde{f} \in \tilde{K}[\tilde{H}]$ is non-zero. Since the field $\tilde{k}$ is infinite and $\tilde{H}$ is supposed to be unirational, the set of $\tilde{k}$-rationals points $\tilde{H}(\tilde{k})$ is Zariski-dense in $\tilde{H}_{\tilde{K}}$. Therefore there exists a $\tilde{k}$-rational point of $\tilde{H}$ on which $\tilde{f}$ does not vanish. Since $\mathcal{H}$ is smooth, we can lift such a point to a point $h \in \mathcal{H}(k^\circ)$ by means of Hensel's Lemma. Clearly $h$ is the point that we were looking for.
\end{proof}

\begin{remark} In the proof of the preceding proposition we showed that $k^\circ[\mathcal{H}]$ is the $k^\circ$-subalgebra of $k[H]$ made of regular functions $f$ such that $|f(g)| \le 1$ for all $h \in \mathcal{H}(k^\circ)$. Adopting the terminology of Bruhat-Tits \cite[D\'efinition 1.7.1]{BruhatTitsII}, one would say that the $k^\circ$-scheme $\mathcal{H}$ is \textit{\'etoff\'e}.
\end{remark}

\subsection{Proof of the Theorem} Let us complete the proof of Theorem \ref{Thm:HyperspecialSubgroupIntro} when $G$ is not quasi-split. Let us recall that if $\mathcal{G}$ is not quasi-split then the residue field $\tilde{k}$ is necessarily infinite (see Proposition \ref{Prop:GroupIsQuasiSplit}). Let us begin with the following technical result:

\begin{lemma} \label{lemma:EnvelopingBoundedSubgroup} Let $H$ be a bounded subgroup of $G(k)$. Then, there are an analytic extension $K$ of $k$ and a faithful representation $\rho \colon G_K \to \GL_{n, K}$ such that
$$ \rho(H) \subset \GL_n(K^\circ). $$
Moreover, if the valuation of $k$ is discrete one can take $K = k$.
\end{lemma}

We postpone the proof of the previous Lemma to the end of the proof of Theorem \ref{Thm:HyperspecialSubgroupIntro}. Let $H$ be a bounded subgroup of $G(k)$ containing $\mathcal{G}(k^\circ)$ and let $K$ and $\rho$ be as in the statement of the previous lemma. Up to extending $K$ we may suppose that $\mathcal{G}_{K^\circ}$ is split.

Since $\rho$ is a closed immersion, the $K$-holomorphically convex envelope of $\rho(H)$ coincides with $\rho(\hat{H}_K)$ (see Proposition \ref{prop:HolomorphicallyConvexEnvelopesClosedImmersions}). Therefore, by the preceding Lemma,
$$ \rho(\hat{H}_K) \subset \widehat{\GL_{n}(K^\circ)}_K = \GL_{n}(K^\circ),$$
where the last equality follows from Proposition \ref{Prop:HolomorphicallyConvexEnvelopeRationalPoints} applied to $\mathcal{H} = \GL_{n, K^\circ}$.  We have therefore the following chain of inclusions :
$$ \mathcal{G}(K^\circ) = \widehat{\mathcal{G}(k^\circ)}_K \subset \hat{H}_K \subset \rho^{-1}(\GL_n(K^\circ)),$$
where the first equality is given by Proposition \ref{Prop:HolomorphicallyConvexEnvelopeRationalPoints} applied with $\mathcal{H} = \mathcal{G}$. Now we can conclude thanks to Theorem \ref{Thm:HyperspecialSubgroupIntro} in the split case: indeed, $\rho^{-1}(\GL_n(K^\circ))$ is a bounded subgroup containing $\mathcal{G}(K^\circ)$ and since $\mathcal{G}_{K^\circ}$ is split by hypothesis, we have
$$ \mathcal{G}(K^\circ) = \hat{H}_K = \rho^{-1}(\GL_n(K^\circ)),$$
which concludes the proof of Theorem \ref{Thm:HyperspecialSubgroupIntro}.
\qed

\bigskip Let us finally prove Lemma \ref{lemma:EnvelopingBoundedSubgroup}:

\begin{proof}[{Proof of Lemma \ref{lemma:EnvelopingBoundedSubgroup}}] Let us first suppose that the valuation of $k$ is discrete and let $\rho_0 \colon G \to \GL_{n, k}$ be any faithful representation and let $\mathcal{E}_0 \df (k^\circ)^n$. Then, the $k^\circ$-submodule of $k^n$,
$$ \mathcal{E} \df \sum_{h \in H} h \cdot \mathcal{E}_0,$$
is bounded (as a subset of $K^n$) because $H$ is bounded. In particular, there exists $\lambda \in k^\times$ such that $\mathcal{E} \subset \lambda \mathcal{E}_0$. Since $k^\circ$ is noetherian, every submodule of $\mathcal{E}_0$ is finitely generated. Thus $\mathcal{E}$ is a torsion-free, finitely generated $k^\circ$-module such that $\mathcal{E} \otimes_{k^\circ} k = k^n$ (it contains $\mathcal{E}_0$). In other words, $\mathcal{E}$ is a lattice of $k^n$ and thus there exists $g \in \GL_n(k)$ such that $g \cdot \mathcal{E} = \mathcal{E}_0$. One concludes by setting $ \rho \df g \rho_0 g^{-1}$. 

If the valuation is not discrete (or, more precisely, if the field $k$ is not maximally complete) some further work is required because of the existence of norms that are not ``diagonalisable''. Let $\rho_0 \colon G \to \GL_{n, k}$ be any faithful representation as before, $K$ a maximally complete extension of $k$ and let us consider the norm on $K^n$,
$$ \| (x_1, \dots, x_n)\|_0 \df \max \{ |x_1|, \dots, |x_n|\}. $$
Since the subgroup $H$ is bounded, the function
$$ \| x \| \df \sup_{h \in H} \| h \cdot x\|_0,$$
is real-valued and it is a norm on $K^n$ verifying the non-archimedean triangle inequality. Since $K$ is maximally complete, there exists a basis $v_1, \dots, v_n$ of $K^n$ and positive real-numbers $r_1, \dots, r_n$ such that
$$ \| x_1 v_1 + \cdots + x_n v_n\| = \max \{ r_1 |x_1| , \dots, r_n |x_n|\}, $$
for all $x_1, \dots, x_n \in K$ \cite[2.4.1 Definition 1 and 2.4.4 Proposition 2]{BGR}. Up to extending further $K$, we may assume that the real numbers $r_1, \dots, r_n$ belong to the value group of $K$. Thus, up to rescaling the basis, we may suppose $r_i = 1$ for all $i$, so that the norm $\| \cdot \|$ is associated with a $K^\circ$-lattice of $K^n$. One finishes the proof as in the discretely-valued case.
\end{proof}

\appendix

\section{Semi-norm associated to an integral model} \label{sec:SeminormsAndIntegralModels} Let $\mathcal{A}$ be a torsion-free $k^\circ$-algebra of finite type and let $A \df \mathcal{A} \otimes_{k^\circ} k$. Since $\mathcal{A}$ is torsion-free, it injects in $A$ and we shall freely consider it as a subset of $A$. For every $f \in A$ we set
$$ \| f \|_\mathcal{A} \df \inf \{ |\lambda| : f/\lambda \in \mathcal{A} \textup{ for all } \lambda \in k^\times \}. $$

\begin{proposition} \label{prop:RationalityOfSeminorms} The semi-norm $\| \cdot \|_\mathcal{A}$ takes values in $|k|$.
\end{proposition}

Since I am not able to point out a suitable reference, we sketch here a proof. Before giving the argument, let us fix some notation. Let $\hat{A}$ be the completion of $A$ with respect to the semi-norm $\| \cdot \|_\mathcal{A}$: we still denote by $\| \cdot \|_\mathcal{A}$ the semi-norm induced on $\hat{A}$. The completion $\hat{\mathcal{A}}$ of $\mathcal{A}$, seen as a $k^\circ$-subalgebra of $\hat{A}$, verifies the following chain of inclusions:
$$ \{ f \in \hat{A} : \| f\|_\mathcal{A} < 1\} \subset \hat{\mathcal{A}} \subset \{ f \in \hat{A} : \| f\|_\mathcal{A} \le 1\}. $$
(\textit{A posteriori}, once we know that the Proposition holds, the second inclusion will be an equality.)

When $\mathcal{A}$ is the ring of polynomials $k^\circ[t_1, \dots, t_n]$, the semi-norm $\|\cdot\|_\mathcal{A}$ is the Gauss norm on polynomials: explicitly, for a polynomial $f$ of the form $\sum_{\alpha \in \N^n} f_\alpha t_1^{\alpha_1} \cdots t_n^{\alpha_n}$, we have
$$ \left\| f \right\|_\mathcal{A} = \max_{\alpha \in \N^n} |f_\alpha|.$$ 
Thus the completion $\hat{A}$ is the so-called \textit{Tate algebra} $k \{ t_1, \dots, t_n \}$ and the semi-norm $\| \cdot \|_\mathcal{A}$ takes values in $|k|$. 

\begin{proof} The statement is trivial if the valuation is discrete, so let us suppose that the valuation is dense. Let 
$ \phi \colon \mathcal{T} = k^\circ[t_1, \dots, t_n] \to \mathcal{A}$
be a surjective homomorphism of $k^\circ$-algebras. We adopt for $\mathcal{T}$ notations similar to the ones for $\mathcal{A}$. The homomorphism $\phi$ induces a surjective\footnote{Because of the equalities $\hat{T} = \hat{\mathcal{T}} \otimes_{k^\circ} k$ and $\hat{A} = \hat{\mathcal{A}} \otimes_{k^\circ} k$, it suffices to show that the induced homomorphism $\phi \colon \hat{\mathcal{T}} \to \hat{\mathcal{A}}$ is surjective.

Let $\lambda \in k$ be a non-zero element such that $|\lambda| < 1$. For every positive integer $n$ let us set $\Lambda_n \df k^\circ / \lambda^n k^\circ$, $ \mathcal{A}_n \df \mathcal{A} \otimes_{k^\circ} \Lambda_n$ and $\mathcal{T}_n \df \mathcal{T} \otimes_{k^\circ} \Lambda_n$. The completion $\hat{\mathcal{T}}$ (resp. $\hat{\mathcal{A}}$) is naturally identified with the projective limit of the $\mathcal{T}_n$'s (resp. of the $\mathcal{A}_n$'s). For every $n$, let $\mathcal{I}_n$ be the kernel of the surjective homomorphism $\mathcal{T}_n \to \mathcal{A}_n$ induced by $\phi$. Then the exact sequence of projective systems,
$$  0 \too (\mathcal{I}_n)_n \too (\mathcal{T}_n)_n \too (\mathcal{A}_n)_n \too 0, $$
satisfies the Mittag-Leffler condition (even better, for every $n$ the map $\mathcal{I}_{n+1} \to \mathcal{I}_n$ is surjective). Therefore, the induced map between projective limits $\hat{\mathcal{T}} \to \hat{\mathcal{A}}$ is surjective. See \cite[Chapter 1, Lemma 3.1 and Exercise 3.15]{Liu}.}
and bounded\footnote{That is, for every $f \in \hat{T}$, we have $\| \phi(f)\|_\mathcal{A} \le \| f\|_\mathcal{T}$.} homomorphism of $k$-Banach algebras, $$ \hat{\phi} \colon \hat{T}  \too \hat{A}.$$

 The open mapping theorem shows that the norm $\| \cdot \|_\mathcal{T}$ attains a minimum on the subset made of elements $g \in \hat{T}$ such that $\hat{\phi}(g) = f$ \cite[1.1.5 Definition 1 and 5.2.7 Theorem 7]{BGR}. If such a minimum is attained in $g_0$, it suffices to show
$$ \| f \|_\mathcal{A} = \| g_0 \|_{\mathcal{T}}.$$

The inequality $\| f\|_{\mathcal{A}} \le \| g_0 \|_{\mathcal{T}}$ is clear because of the boundedness of the homomorphism $\hat{\phi}$. Let us suppose by contradiction $\| f\|_{\mathcal{A}} < \| g_0 \|_{\mathcal{T}}$. Up to rescaling $g_0$ we may suppose $\| g_0 \|_{\mathcal{T}} = 1$ (it is crucial here $\| \cdot \|_{\mathcal{T}}$ takes values in $|k|$). By density of the valuation, there exists $\lambda \in k$ such that $|\lambda| > 1$ and $\| \lambda f\|_\mathcal{A} < 1$ hence $\lambda f$ belongs to $\mathcal{A}$. Since $\hat{\phi}$ is surjective, there exists $g_1 \in \hat{\mathcal{T}}$ such that $\phi(g_1) = \lambda f$. Therefore, $\phi(g_1 / \lambda ) = f$ and
$$ \| g_1 / \lambda \|_\mathcal{T} < \| g_1 \|_\mathcal{T} \le 1,$$
contradicting the minimality of $g_0$.
\end{proof}

\small

\providecommand{\bysame}{\leavevmode\hbox to3em{\hrulefill}\thinspace}
\providecommand{\MR}{\relax\ifhmode\unskip\space\fi MR }
\providecommand{\MRhref}[2]{%
  \href{http://www.ams.org/mathscinet-getitem?mr=#1}{#2}
}
\providecommand{\href}[2]{#2}


\begin{thebibliography}{RTW10}

\bibitem[BGR84]{BGR}
S.~Bosch, U.~G{\"u}ntzer, and R.~Remmert, \emph{Non-{A}rchimedean analysis},
  Grundlehren der Mathematischen Wissenschaften [Fundamental Principles of
  Mathematical Sciences], vol. 261, Springer-Verlag, Berlin, 1984, A systematic
  approach to rigid analytic geometry. \MR{746961 (86b:32031)}

\bibitem[BLR90]{NeronModels}
S.~Bosch, W.~L{\"u}tkebohmert, and M.~Raynaud, \emph{N\'eron models},
  Ergebnisse der Mathematik und ihrer Grenzgebiete (3) [Results in Mathematics
  and Related Areas (3)], vol.~21, Springer-Verlag, Berlin, 1990.

\bibitem[Bor91]{Borel}
A.~Borel, \emph{Linear algebraic groups}, second ed., Graduate Texts in
  Mathematics, vol. 126, Springer-Verlag, New York, 1991.

\bibitem[Bri05]{Brion}
M.~Brion, \emph{Lectures on the geometry of flag varieties}, Topics in
  cohomological studies of algebraic varieties, Trends Math., Birkh\"auser,
  Basel, 2005, pp.~33--85.

\bibitem[Bru62]{BruhatBruxelles}
F.~Bruhat, \emph{Sur les sous-groupes compacts maximaux des groupes
  semi-simples {${\frak{p}}$}-adiques}, Colloq. {T}h\'eorie des {G}roupes
  {A}lg\'ebriques ({B}ruxelles, 1962), Librairie Universitaire, Louvain;
  Gauthier-Villars, Paris, 1962, pp.~69--76.

\bibitem[Bru95]{BruhatBourbaki}
\bysame, \emph{Sous-groupes compacts maximaux des groupes semi-simples {$\frak
  p$}-adiques}, S\'eminaire {B}ourbaki, {V}ol.\ 8, Soc. Math. France, Paris,
  1995, pp.~Exp.\ No.\ 271, 413--423.

\bibitem[BTI]{BruhatTitsI}
F.~Bruhat and J.~Tits, \emph{Groupes r\'eductifs sur un corps local}, Inst.
  Hautes \'Etudes Sci. Publ. Math. (1972), no.~41, 5--251.

\bibitem[BTII]{BruhatTitsII}
\bysame, \emph{Groupes r\'eductifs sur un corps local. {II}. {S}ch\'emas en
  groupes. {E}xistence d'une donn\'ee radicielle valu\'ee}, Inst. Hautes
  \'Etudes Sci. Publ. Math. (1984), no.~60, 197--376.

\bibitem[Cona]{ConradUnirationality}
B.~Conrad, \emph{{L}ang's theorem and unirationality}, available at
  \texttt{http://math.stanford.}
  \texttt{edu/\textasciitilde{}conrad/249CS13Page/handouts/langunirat.pdf}.

\bibitem[Conb]{ConradSmoothRep}
\bysame, \emph{{S}mooth representations and {H}ecke algebras for $p$-adic
  groups}, available at
  \texttt{http://math.stanford.edu/\textasciitilde{}conrad/JLseminar/Notes/L2.pdf}.

\bibitem[Con14]{ConradSGA3}
\bysame, \emph{Reductive group schemes}, Autour des sch\'emas en groupes,
  \'Ecole d'\'et\'e ``Sch\'emas en groupes'', Group Schemes, A celebration of
  SGA3, Volume I (Brochard S., Conrad B., and Oesterl\'e J., eds.), Panoramas
  et synth\`eses, vol. 42-43, Soci\'et\'e Math\'ematique de France, 2014.

\bibitem[SGA3]{SGA3}
P.~Gille and P.~Polo (eds.), \emph{Sch\'emas en groupes ({SGA} 3). {T}ome
  {III}. {S}tructure des sch\'emas en groupes r\'eductifs}, Documents
  Math\'ematiques (Paris), 8, Soci\'et\'e Math\'ematique de France, Paris,
  2011, S{\'e}minaire de G{\'e}om{\'e}trie Alg{\'e}brique du Bois Marie
  1962--64, A seminar directed by M. Demazure and A. Grothendieck with the
  collaboration of M. Artin, J.-E. Bertin, P. Gabriel, M. Raynaud and J-P.
  Serre, Revised and annotated edition of the 1970 French original.

\bibitem[EGA2]{EGA2}
A.~Grothendieck, \emph{\'{E}l\'ements de g\'eom\'etrie alg\'ebrique. {II}.
  \'{E}tude globale \'el\'ementaire de quelques classes de morphismes}, Inst.
  Hautes \'Etudes Sci. Publ. Math. (1961), no.~8, 222. \MR{0163909 (29 \#1208)}

\bibitem[Gro64]{SGA3ExpXI}
\bysame, \emph{Crit\`eres de repr\'esentabilit\'e. {A}pplications aux
  sous-groupes de type multiplicatif des sch\'emas en groupes affines},
  Sch\'emas en {G}roupes ({S}\'em. {G}\'eom\'etrie {A}lg\'ebrique, {I}nst.
  {H}autes \'{E}tudes {S}ci., 1963/64) {F}asc. 3, Inst. Hautes \'Etudes Sci.,
  Paris, 1964, pp.~Expos\'e 11, 53.

\bibitem[Hij75]{Hijikata}
H.~Hijikata, \emph{On the structure of semi-simple algebraic groups over
  valuation fields. {I}}, Japan J. Math. (N.S.) \textbf{1} (1975), no.~2,
  225--300.

\bibitem[Kem78]{Kempf}
G.~Kempf, \emph{The {G}rothendieck-{C}ousin complex of an induced
  representation}, Adv. in Math. \textbf{29} (1978), no.~3, 310--396.

\bibitem[Liu02]{Liu}
Q.~Liu, \emph{Algebraic geometry and arithmetic curves}, Oxford Graduate Texts
  in Mathematics, vol.~6, Oxford University Press, Oxford, 2002, Translated
  from the French by Reinie Ern{\'e}, Oxford Science Publications.

\bibitem[PR94]{PlatonovRapinchuk}
V.~Platonov and A.~Rapinchuk, \emph{Algebraic groups and number theory}, Pure
  and Applied Mathematics, vol. 139, Academic Press, Inc., Boston, MA, 1994,
  Translated from the 1991 Russian original by Rachel Rowen.

\bibitem[Rou77]{RousseauThese}
G.~Rousseau, \emph{Immeubles des groupes r\'educitifs sur les corps locaux},
  U.E.R. Math\'ematique, Universit\'e Paris XI, Orsay, 1977, Th{\`e}se de
  doctorat, Publications Math{\'e}matiques d'Orsay, No. 221-77.68.

\bibitem[RTW10]{RemyThuillierWerner}
B.~R{\'e}my, A.~Thuillier, and A.~Werner, \emph{Bruhat-{T}its theory from
  {B}erkovich's point of view. {I}. {R}ealizations and compactifications of
  buildings}, Ann. Sci. \'Ec. Norm. Sup\'er. (4) \textbf{43} (2010), no.~3,
  461--554.

\bibitem[Ser94]{CohomologieGaloisienne}
J-P. Serre, \emph{Cohomologie galoisienne}, fifth ed., Lecture Notes in
  Mathematics, vol.~5, Springer-Verlag, Berlin, 1994.

\bibitem[Spr98]{Springer}
T.~A. Springer, \emph{Linear algebraic groups}, second ed., Progress in
  Mathematics, vol.~9, Birkh\"auser Boston, Inc., Boston, MA, 1998.

\end{thebibliography}
\end{document}